\documentclass[11pt]{amsart}

\usepackage{amssymb}
\usepackage{bbm,titletoc} 
\usepackage[normalem]{ulem}
\usepackage[dvips]{graphicx}
\usepackage{float,enumitem}

\usepackage[usenames,dvipsnames]{color}

\usepackage{tikz}
\usetikzlibrary {positioning}
\definecolor {processblue}{cmyk}{0.96,0,0,0}

\makeatletter
\def\moverlay{\mathpalette\mov@rlay}
\def\mov@rlay#1#2{\leavevmode\vtop{%
   \baselineskip\z@skip \lineskiplimit-\maxdimen
   \ialign{\hfil$#1##$\hfil\cr#2\crcr}}}

\makeatother

\numberwithin{equation}{section}

\definecolor{Ruta}{rgb}{0.309, 0.459, 0.208}
\definecolor{Vino}{rgb}{0.256,0,0}
\definecolor{Siva}{rgb}{0.116,0.116,0.116}
\definecolor{Siva'}{rgb}{0.250,0.116,0.116}

\usepackage{amsfonts}
\usepackage{amsmath}
\usepackage{verbatim}

\usepackage[pagebackref,colorlinks,linkcolor=red,citecolor=OliveGreen,urlcolor=blue,hypertexnames=true]{hyperref}

\usepackage{blkarray}
\usepackage{multirow}
\usepackage{xspace} 
\usepackage{multirow,bigdelim}

\newcommand{\cA}{\mathcal{A}}

\def\beq{ \begin{equation} }
\def\eeq{ \end{equation} }

\def\bep{\begin{proof}}
\def\eep{\end{proof}}

\def\ben{ \begin{enumerate} }
\def\een{ \end{enumerate} }

\newcommand{\id}{{\rm id}}

\newcommand{\ol}{\overline}

\DeclareMathOperator{\vspan}{span}

\allowdisplaybreaks[1]

\newtheorem{theorem}{Theorem}[section]
\newtheorem{proposition}[theorem]{Proposition}

\newtheorem{corollary}[theorem]{Corollary}

\theoremstyle{definition}

\newtheorem{lemma}[theorem]{Lemma}

\newtheorem{remark}[theorem]{Remark}

\newcommand{\s}{\sigma} 
\newcommand{\CC}{\mathbb{C}}
\newcommand{\RR}{\mathbb{R}}

\newcommand{\N}{\mathbb{N}}
\newcommand{\C}{C}

\newcommand{\F}{\mathbb{F}}

\newcommand{\la}{\langle}
\newcommand{\ra}{\rangle}

\newcommand{\st}{\;\ifnum\currentgrouptype=16 \middle\fi|\;}

\usepackage{graphicx}
\usepackage{calc}
\newlength{\depthofsumsign}
\makeatletter
\newcommand*{\rom}[1]{\expandafter\@slowromancap\romannumeral #1@}
\makeatother

\title[The length of a generic vector subspace of $M_n(\F)$]{Sweeping words and the length of a\\ generic vector subspace of $M_n(\F)$}

\author[I. Klep]{Igor Klep${}^{1}$}
\address{Igor Klep, Department of Mathematics, 
The University of Auckland, New Zealand}
\email{igor.klep@auckland.ac.nz}
\thanks{${}^1$Supported by the Marsden Fund Council from Government funding, administered by the Royal Society of New Zealand.
Partially supported by the Slovenian Research Agency grants P1-0222, L1-4292
and L1-6722. Part of this research was done while the author was on leave from the University of Maribor.}

\author[\v S. \v Spenko]{\v Spela \v Spenko${}^{2}$}
\address{\v Spela \v Spenko,  Faculty of Mathematics and Physics,  University of Ljubljana, Slovenia} \email{spela.spenko@fmf.uni-lj.si}
\thanks{${}^{2}$Supported by the Slovenian Research Agency and the L'Or\'eal-UNESCO scholarship ``For women in science".}

\date{\today}

\begin{document}

\begin{abstract}
The main result of this short note is a generic version of Paz's conjecture
on the lengths of generating sets in matrix algebras. 
Consider a generic $g$-tuple $A=(A_1,\ldots, A_g)$ of $n\times n$ matrices over a field. We show that
whenever $g^{2d}\geq n^2$, the set of all words of degree $2d$ in $A$ spans the full $n\times n$ matrix algebra. 
Our proofs use generic matrices, are  combinatorial and depend on the construction of a special kind of directed multigraphs with few edge-disjoint walks.
\end{abstract}

\maketitle
\section{Introduction}

Let $\F$ be a field and let $\cA$ be an associative $\F$-algebra.
Given a generating set $1\in S$ of $\cA$ let $S^k$ denote the set of all products of the form $s_1\cdots s_k$ with $s_i\in S$. 
If 
\[
\vspan S^{\ell-1} \subsetneq 
\vspan S^\ell=\cA,
\]
we say that
$S$ has (generating) {\bf length} $\ell$.
These lengths feature prominently in the
study of growth of algebras and the
Gelfand-Kirillov dimension \cite{KL,BK}.

A fundamental problem is to find bounds on the length of
generating sets. Much activity has focused on $\cA=M_n(\F)$
(see e.g.~\cite{Paz, FGG, Pap, LoRo, Ros}), 
where the best known bound on the length of generating sets is
$O(n^{3/2})$, due to Pappacena \cite{Pap}. The Paz conjecture \cite{Paz} states that the bound is $2n-2$, and it is easy to see that this bound would be sharp.
We refer the reader to \cite{LS} for the study of an analogous problem in groups.

In this short note
we establish a strengthening of the Paz conjecture 
in a generic setting: 
the length of a \emph{generic}
generating set in $M_n(\F)$ is $O(\log n)$ (Corollary  \ref{generic}).
To prove this bound we establish the 
existence of \emph{sweeping} words $w_1,\ldots,w_{n^2}$
of degree $2\lceil \log_g n\rceil$
in $g$ freely noncommuting letters $x_1,\ldots, x_g$. That is,
there exist (symmetric) $n\times n$ matrices $A_1,\ldots, A_g$
such that $w_1(A_1,\ldots,A_g),\ldots, w_{n^2}(A_1,\ldots,A_g)$ span $M_n(\F)$; see Theorem \ref{thm:words}.
Here is a simple corollary. Given 
$g^{2d}\geq n^2$ consider 
the set of all words $w$ of length $2d$ in $g$ matrices $A\in M_n(\F)^g$. 
Vectorize each matrix $w(A)$ and arrange these vectors into a matrix (of size $n^2\times g^{2d}$). Corollary \ref{Ros} shows that 
this matrix is generically  of full rank, generalizing a theorem of Rosenthal \cite{Ros} who established the special case $d=1$.

The key step in the proofs is the construction of 
a special kind of  directed multigraphs with few edge-disjoint walks (Subsection \ref{sec:3.1}).
Another ingredient going into our proofs are
generic matrices and their properties \cite{Row,GRZ}.

The paper is organized as follows. Section \ref{sec:2}
gives notation, preliminaries,
and presents our main results
on sweeping words and lengths in matrix algebras. 
Section \ref{sec:proofs} gives proofs of our results,
including the graph-theoretic construction in Subsection \ref{sec:3.1}.

\subsection*{Acknowledgments}
The authors thank Benoit Collins for several stimulating
discussions, 
Jason P.~Bell for sharing his expertise, and  
Jurij Vol\v ci\v c for carefully reading a preliminary version of the manuscript.

\section{Main results}\label{sec:2}

\subsection{Notation and preliminaries}
Let $\F$ denote an infinite field and $M_n(\F)$ the algebra of $n\times n$ matrices over $\F$. 
We denote the free associative algebra generated by $x_1,\dots,x_g$ by $\F\langle x_1,\dots,x_g\rangle$. 
By $\langle x_1,\dots,x_g\rangle$ we denote the free monoid generated by $x_1,\dots,x_g$, and by $\langle x_1,\dots,x_g\rangle_d$ words in $\langle x_1,\dots,x_g\rangle$ of length $d$. In case $g=2$  we write $x,y$ instead of $x_1,x_2$. The set $\{1,\dots,d\}$ is denoted by $\N_d$.

\subsubsection{Generic matrices and the discriminant}\label{gmn} 
We denote by $\C=\F[x_{ij}^{(k)}\mid 1\leq i,j\leq n, 1\leq k\leq g]$  a commutative polynomial algebra. 
The elements $X_k=(x_{ij}^{(k)})\in M_n(\C)$, $1\leq k\leq g$, are called {\bf generic matrices}. 
The {\bf discriminant} $\Delta(A^{(1)},\dots,A^{(n^2)})$ of $n\times n$ matrices $A^{(1)},\dots,A^{(n^2)}$ is the determinant of the $n^2\times n^2$ matrix whose $k$-th column $v^{(k)}$ is 
the vectorized matrix $A^{(k)}$; i.e., $v^{(k)}_{(n-1)i+j}=A^{(k)}_{ij}$. 

\subsubsection{Locally linearly independent words}
We say that $w_1,\dots,w_m\in \langle x_1,\dots,x_g\rangle$ are {\bf $M_n(\F)$-locally linearly independent} if $w_1(A),\dots,w_m(A)$ are linearly independent for some $A\in M_n(\F)^g$. 
This concept first appeared in \cite{CHSY}, later it was studied algebraically in \cite{BrKl}, and recently in \cite{BPS}. 

Since we assume that $\F$ is infinite we have the following easy observation.

\begin{lemma}\label{cap}
Words $w_1,\dots,w_{n^2}\in \langle x_1,\dots,x_g\rangle$ are $M_n(\F)$-locally linearly independent if and only if the discriminant of $w_1(X_1,\dots,X_g),\dots,w_{n^2}(X_1,\dots,X_g)$ is nonzero.
\end{lemma}

\subsection{Main result on words}

\begin{theorem}\label{thm:words}
Let $g\geq 2$ and $d=\lceil{\log_g n}\rceil$. Then there exist 
$M_n(\F)$-locally linearly independent
words
$w_1,\dots,w_{n^2}\in \langle x_1,\dots,x_g\rangle_{2d}$. 
That is, for some $A\in M_n(\F)^g$ the matrices $w_1(A),\dots,w_{n^2}(A)$ are linearly independent and thus span $M_n(\F)$.
\end{theorem}


\subsubsection{Length of a vector space}
Let $V$ be a vector subspace of $M_n(\F)$. By $V^k$ we denote the vector space spanned by the words of length at most $k$ evaluated at $V$. 
The {\bf length} of $V$ is the integer $\ell$ yielding a stationary chain
$$V\subsetneq V^2 \subsetneq \dots \subsetneq V^\ell=V^{\ell +1}.$$ 
We say that words $w_1,\dots,w_t\in \langle x_1,\dots,x_g\rangle$ {\bf sweep $M_n(\F)$}
if there exists  $A\in M_n(\F)^g$ such that $w_1(A),\dots, w_t(A)$ span $M_n(\F)$.
Given a subset $S$ of $M_n(\F)^g$, a vector space $V\subseteq M_n(\F)$ of dimension $g$ 
is {\bf $S$-generic} if it can be spanned by elements $A_1,\dots,A_g$ satisfying $(A_1,\dots,A_g)\in S$.

\begin{corollary}[Generic version of Paz's conjecture]\label{generic}
Let $g\geq 2$ and let $\F$ be an infinite field. 
There exists a nonempty Zariski  open subset $S\subseteq M_n(\F)^g$ such that the length of an $S$-generic  vector subspace of $M_n(\F)$ is of order $O(\log n)$. 
\end{corollary}
Note that if $\F\in\{\RR,\CC\}$, then a nonempty Zariski open subset of $\F^{m}$ is automatically dense in the Euclidean topology.

\subsubsection{Words in random matrices span the full matrix algebra}

By Corollary \ref{generic}, given a $g$-tuple $A$ of random $n\times n$ matrices,
words of length $O(\log n)$ in $A$ span $M_n(\F)$. In particular, we have:

\begin{corollary}\label{Ros}
For each $g$ satisfying $n^2\leq g^{2d}$ there exists a set of $g$ matrices such that words of length $2d$ in those matrices span $M_n(\F)$. 
\end{corollary}

Corollary \ref{Ros} partially answers a question posed in \cite{Ros}, where this result is established in the case $d=1$. 
The answer is complete for $n\in \N$ satisfying $g^{2d-1}<n^2\leq g^{2d}$.  
The question whether the words of length $2d-1$ sweep $M_n(\F)$ in the case $n^2\leq g^{2d-1}$ remains.

\section{Proofs}
\label{sec:proofs}

The proof of Theorem \ref{thm:words} reduces to the study of a special kind of graphs which we introduce in the first subsection. 
The main ideas can be revealed already in the case of two variables, and the general case is only notationally more difficult, so we focus on $g=2$. 
In the next  proposition we state for convenience this special case separately.

\begin{proposition}\label{prop}
Let $d=\lceil{\log_2n}\rceil$. There exist  $M_n(\F)$-locally linearly independent words $w_1,\dots,w_{n^2}\in \la x,y\ra_{2d}$. 
\end{proposition}

\subsection{Graphs}\label{sec:3.1}
We recursively construct  a family of graphs $(G_d)_{d\in \N}$. 
Let $G_0$ be a graph with one vertex labeled by $1$ and no edges. 
We let $G_d$ be a (directed) graph with $2^d$ vertices labeled by $\N_{2^d}$, and define the edges as follows. There is a (directed) edge from $i$ to $j$  in $G_d$ of multiplicity $4e$, $1\leq i,j\leq {2^{d-1}}$,  if there is an edge of multiplicity $e$ from $i$ to $j$ in $G_{d-1}$. 
Moreover, each vertex $i$, $1\leq i\leq 2^{d-1}$, has additionally $2^{d+1}$ loops, and there are $2^d$ edges from $i$ to $i+2^{d-1}$ and back for $1\leq i\leq 2^{d-1}$. We label the loops by $x$, and other edges by $y$. 
Note that $G_d$ contains $2^{2d+1}d$ edges, half of them labeled by $x$ and the other half by $y$.

For example, the figures below show $G_1$ and $G_2$ with the numbers on edges corresponding to their respective multiplicities, and instead of the labels $x$ and $y$ we use dashed (resp. solid) edges for the edges corresponding to $x$ (resp. $y$).
\vspace{-2cm}
\begin{figure}[H]
\tikzset{every loop/.style={min distance=30mm,looseness=10,in=-60,out=60}}
\begin{center}
\begin {tikzpicture}[-latex ,auto ,node distance =4 cm and 5cm ,on grid ,
semithick ,
state/.style ={ circle ,color =gray ,
draw , text=black, minimum width =.01 cm}]
\node[state] (A)  {$2$};
\node[state] (B) [right of= A] {$1$};
\path (B) edge [dashed,loop right] node[right] {$4$} (B);
\path (A) edge [bend left =25] node[above ] {$2$} (B);
\path (B) edge [bend right = -25] node[below] {$2$} (A);
\end{tikzpicture}
\end{center}
\vspace{-1.5cm}
\caption{$G_1$}
\end{figure}

\begin{figure}[H]
\tikzset{every loop/.style={min distance=30mm,looseness=10,in=45,out=135}}
\begin{center}
\begin {tikzpicture}[-latex ,auto ,node distance =4 cm and 5cm ,on grid ,
semithick ,
state/.style ={ circle ,color =gray ,
draw , text=black, minimum width =.01 cm}]
\node[state] (A)  {$3$};
\node[state] (B) [right of= A] {$1$};
\node[state] (C) [right of= B] {$2$};
\node[state] (D) [right of= C] {$4$};
\path (B) edge [dashed,loop right] node[above] {$24$} (B);
\path (C) edge [dashed,loop right] node[above] {$8$} (C);
\path (A) edge [bend left =25] node[above ] {$4$} (B);
\path (B) edge [bend right = -25] node[below ] {$4$} (A);
\path (B) edge [bend left =25] node[above ] {$8$} (C);
\path (C) edge [bend right = -25] node[below ] {$8$} (B);
\path (C) edge [bend left =25] node[above ] {$4$} (D);
\path (D) edge [bend right = -25] node[below ] {$4$} (C);
\end{tikzpicture}
\end{center}
\caption{$G_2$}
\end{figure}

If $p$ is a walk in the  graph $G_d$ then we associate to it a word corresponding to the labels on the edges passed by $p$ in the respective order.

We now state a technical lemma that will be used extensively in the proof of Proposition \ref{prop}.
\begin{lemma}\label{graf}
There is a unique way of partitioning the graph $G_d$ in $2^{2d}$ (edge-disjoint) walks $p_{ij}$ of length $2d$, $1\leq i,j\leq 2^d$, such that $p_{ij}$ starts at $i$ and ends at $j$, which yield all the words in $\langle x,y\rangle_{2d}$. 
\end{lemma}

\begin{proof}
We prove the lemma by induction on $d$. Let us denote by $G_d^{(m)}$ the graph obtained from $G_d$ by multiplying the  multiplicity of each edge by $m$. 

We claim that $G_d^{(m)}$ can be in only one way partitioned into $m2^{2d}$ walks of length $2d$ such that $m$ walks start at $i$ and end at $j$, $1\leq i,j\leq 2^{d}$, and such that each word in $\langle x,y\rangle_{2d}$  corresponds to $m$ walks. 
Consider first $G_1^{(m)}$. Then the only way of obtaining the desired partition is to take $m$ walks $\{2\to 1,1\to 2\}$, $m$ walks $\{1\to 1,1\to 2\}$, $m$ walks $\{2\to 1,1\to1\}$ and $m$ walks $\{1\to 1,1\to 1\}$ as can be easily seen.
Suppose that  the claim holds for all graphs $G_\ell^{(m)}$,  $\ell<d$.  
Consider now $G_d^{(m)}$. Since there are no loops on the vertices labeled by $i$, $2^{d-1}+1\leq i\leq 2^d$, all words with the starting or ending point in these vertices need to begin, respectively end, with $y$. 
By the condition on the partition there are exactly half of walks with this property, thus words with other starting, resp. ending, point need to begin, resp. end, with $x$. 
Removing the edges starting or ending at $i$, $2^{d-1}+1\leq i\leq 2^d$, and $m2^{d+1}$ loops on vertices $i$, $1\leq i\leq 2^{d-1}$, 
we obtain a graph on $2^{d-1}$ vertices labeled by $\N_{2^{d-1}}$ (ignoring the isolated points) which coincides with $G_{d-1}^{(4m)}$ by construction, 
and which we need to partition into $4m2^{2(d-1)}$ walks of length $2(d-1)$ 
(as we have already removed the starting and the ending edge of walks in $G_{d}^{(m)}$) 
such that $4m$ walks start at $i$ and end at $j$, $1\leq i,j\leq 2^{d-1}$, 
and each word in $\langle x,y\rangle_{2d-2}$ corresponds to $4m$ walks. By the induction hypothesis, there is only one such a partition. 
The lemma thus follows by taking $m=1$. 
\end{proof}

For the proof of Theorem \ref{thm:words} we will need a slight generalization of the previous lemma. We thus introduce a graph $G^{g}_d$ which has $g^d$ vertices and is defined recursively by setting $G^g_0$ to be the graph with $1$ vertex labeled by $1$ and no edges. 
Having constructed $G^g_{d-1}$ we let $G^g_d$ be a directed graph with $g^d$ vertices labeled by $\N_{g^d}$, and having a (directed) edge from $i$ to $j$ of multiplicity $g^2e$, $i,j\in \N_{g^{d-1}}$, if there is an edge of multiplicity $e$ from $i$ to $j$ in $G^g_{d-1}$, and is labeled as the corresponding edge in $G^g_{d-1}$, 
and there are $g^d$ edges from $i$ to $i+(k-1)g^{d-1}$ and back for $1\leq i\leq g^{d-1}$, labeled by $x_k$, $1\le k\leq g$  
(for $k=1$ every loop has multiplicity $g^{2d}$).  
Note that $G^g_d$ contains $2dg^{2d}$ edges.

\begin{lemma}\label{graf-g}
There is a unique partition of the graph $G^g_d$ in $g^{2d}$ (edge-disjoint) walks $p_{ij}$ of length $2d$, $1\leq i,j\leq 2^{d}$, such that $p_{ij}$ starts at $i$ and ends at $j$, which yield all the words in $\langle x_1,\dots,x_g\rangle_{2d}$.
\end{lemma}

The proof of Lemma \ref{graf-g} is a straightforward modification of Lemma \ref{graf} and is omitted.

\subsection{Proof of Theorem \ref{thm:words}}
\begin{proof}[Proof of Proposition \ref{prop}]
By Lemma \ref{cap} we need to show that 
$$p(x_{11},\dots,x_{nn},y_{11},\dots,y_{nn}):=\Delta(w_1(X,Y),\dots,w_{n^2}(X,Y))$$
is nonzero for some $w_1,\dots,w_{n^2}\in \langle x,y\rangle_{2d}$, 
 where $X=(x_{ij})$, $Y=(y_{ij})$ are generic $n\times n$ matrices. 
(We may and we will assume that $X$ is diagonal \cite[Proposition 1.3.15]{Row}.) 
By the definition of the discriminant,
\beq\label{cap-ev}
p(x_{11},\dots,x_{nn},y_{11},\dots,y_{nn})=\sum_{\s\in S_{n^2}}(-1)^\s \prod_{1\leq i,j\leq n}w_{\s(k_{ij})}(X,Y)_{ij},
\eeq
where $k_{ij}=(i-1)n+j$, and $w_{k}(X,Y)_{ij}$ denotes the commutative polynomial at the entry $(i,j)$ of the word $w_k$ evaluated at the generic matrices $X,Y$.

Let us define the lexicographic order on $\langle x,y\rangle$ with $x>y$ and denote by $v_s$ the vector of the 
words of length $s$  listed  decreasingly with respect to this order. 
Let 
\beq\label{a_n}
a^{(n)}=e_nv_dv_d^te_n=
\left(\begin{array}{ll}
xv_{d-1}v_{d-1}^tx&xv_{d-1}v_{d-1}^tye_{n'}\\
e_{n'}yv_{d-1}v_{d-1}^tx&e_{n'}yv_{d-1}v_{d-1}^tye_{n'}
\end{array}\right) 
\eeq
be the block matrix consisting of words, where $n'=n-2^{d-1}$ and $e_{n'}=e_{11}+\dots +e_{n'n'}$, with blocks of the size $2^{d-1}\times 2^{d-1}$, $2^{d-1}\times n'$, $n'\times 2^{d-1}$, and $n'\times n'$, respectively. 

We proceed to find a monomial that appears in the product on the right-hand side of \eqref{cap-ev} for a unique $\s\in S_{n^2}$. 
Let $2^{d-1}<n\leq 2^d$ and $x_i=x_{ii}$. 
We define 
$$
\ol m_{n}^{(ij)}=  \left\{\begin{array}{ll}
  x_ix_j&\quad \text{if $1\leq i,j\leq 2^{d-1}$},\\
  x_i y_{j-2^{d-1},j}&\quad\text{if $1\leq i\leq 2^{d-1},2^{d-1}< j\leq n$},\\
  x_jy_{i,i-2^{d-1}}&\quad\text{if $2^{d-1}< i\leq n,1\leq j\leq 2^{d-1}$},\\
  y_{i,i-2^{d-1}}y_{j-2^{d-1},j}&\quad \text{if $2^{d-1}< i,j\leq n$}.\\
  \end{array}\right.
$$
We further inductively define
$$
m_1^{(11)}=1,\quad m_{n}^{(ij)}=m_{2^{d-1}}^{(i_d,j_d)}\ol m_{n}^{(ij)},
$$
where $i_d\equiv i\mod 2^{d-1}$,  $j_d\equiv j \mod 2^{d-1}$, $1\leq i_d,j_d\leq 2^{d-1}$.
Consider 
the monomial  defined  by 
$$m_n=\prod_{1\leq i,j\leq n} m_{n}^{(ij)}
.$$
In particular, in the case $n=2^d$ we have 
$$m_{2^d}=(m_{2^{d-1}})^4\prod_{1\leq i,j\leq 2^d}\ol m_{2^d}^{(ij)}.$$

Let $w_{(i-1)n+j}=a^{(n)}_{ij}$ for $a^{(n)}$ defined in \eqref{a_n}.  
We claim that $m_n$ appears in  
$$
P_n^\s=\prod_{1\leq i,j\leq n}w_{\s(k_{ij})}(X,Y)_{ij}
$$
only for $\s=\id$. 
 By the construction of $m_n^{(ij)},m_n$ and \eqref{a_n}, 
$m_n^{(ij)}$ has a nonzero coefficient in the commutative polynomial $a^{(n)}_{ij}(X,Y)_{ij}$ and thus 
the same holds for  the monomial $m_n$ in
$$\prod_{1\leq i,j\leq n}w_{k_{ij}}(X,Y)_{ij}=\prod_{1\leq i,j\leq n}a^{(n)}_{ij}(X,Y)_{ij}.$$
It remains to show that $m_n$  does not appear in $P_n^\sigma$ for $\s\neq \id$. For this we use a graph-theoretic language.

We first consider the case $n=2^d$. 
We can present the monomial $m_n$ as a walk in a graph on $n$ vertices, 
in which there is a directed edge of multiplicity $s_{ij}$ between vertices $i$ and $j$ labeled by $y$ if $s_{ij}$ is the degree of $y_{ij}$ in the monomial $m_n$, 
and there are $s_i$ loops on the vertex $i$ labeled by $x$ if $s_i$ is the degree of $x_i$ in $m_n$. Since $m_n$ needs to be written as a product of $n^2$ monomials $u_{ij}$,  
$1\leq i,j\leq n$, 
arising from monomials in $w_{\s(k_{ij})}(X,Y)_{ij}$, $w_{k_{ij}}\in\langle x,y\rangle_{2d}$, 
our problem reduces to finding partitions of the graph associated to $m_n$ into $n^2$ walks $p_{ij}$, $1\leq i,j\leq n$, of length $2d$  that yield all the words in $\langle x,y\rangle_{2d}$. 
Lemma \ref{graf} asserts that there is only one such partition, and thus concludes the proof in the case $n=2^d$. 

For arbitrary $n$ we observe that if 
$$\sum_{\s\in S_{n^2}} (-1)^\s \prod_{1\leq i,j\leq n}a^{(n)}_{\s((i-1)n+j)}(X,Y)_{ij}$$
equaled $0$ then $m_{n}$ would appear in $P_n^{\rm\small id}$ and in some other 
$P_n^\rho$
with $\rho\in S_{n^2}$. 
As $m_n^{(ij)}$ can be identified  with $m_{2^d}^{(ij)}$ for $1\leq i,j\leq n$, $m_{2^d}$ would have a nonzero coefficient in the product $P_{2^d}^\s$ for $\s=\id$ and $\s=\tilde\rho\in S_{2^d}$ (here $\tilde\rho$ is the  permutation in $S_{2^d}$ induced by $\rho$ and fixing all $i>n$), which is impossible by the claim proved in the previous paragraph. 
Thus the words $a^{(n)}_{11},\dots,a^{(n)}_{nn}$ are $M_n(\F)$-locally linearly independent.
\end{proof}

\begin{proof}[Proof of Theorem \ref{thm:words}]
One follows the steps of the proof of Proposition \ref{prop} where initially one needs to consider the case $n=g^d$, and defines the monomial $m_n$ inductively  corresponding to the matrix 
\beq\label{a_gn}
a^{(n)}=e_nv_dv_d^te_n=
\left(\begin{array}{cccc}
x_1v_{d-1}v_{d-1}^tx_1&\dots&x_1v_{d-1}v_{d-1}^tx_{g-1}&x_1v_{d-1}v_{d-1}^tx_ge_{n'}\\
\vdots&\ddots&\vdots&\vdots\\
x_{g-1}v_{d-1}v_{d-1}^tx_1&\dots&x_{g-1}v_{d-1}v_{d-1}^tx_{g-1}&x_{g-1}v_{d-1}v_{d-1}^tx_ge_{n'}\\
e_{n'}x_gv_{d-1}v_{d-1}^tx_1&\dots&e_{n'}x_gv_{d-1}v_{d-1}^tx_{g-1}&e_{n'}x_gv_{d-1}v_{d-1}^tx_ge_{n'}
\end{array}\right), 
\eeq
where $g^{d-1}<n\leq g^d$, $n'=n-g^{d-1}$, and $v_s$ denotes the vector of $g^s$ words of length $s$ listed decreasingly in the monomial order induced by setting $x_1>\dots >x_g$. Instead of applying Lemma \ref{graf} one concludes the proof by applying Lemma \ref{graf-g}.
\end{proof}

\begin{proof}[Proof of Corollary \ref{generic}]
Let $w_1,\dots,w_{n^2}$ be the $M_n(\F)$-locally linearly independent words of length $2d=2\lceil{\log_g n}\rceil$ whose existence was established in Theorem \ref{thm:words}. 
As then the discriminant $\Delta:=\Delta(w_1(X_1,\dots,X_g),\dots,w_{n^2}(X_1,\dots,X_g))$ is nonzero, the subset $S$ of $A\in M_n(\F)^g$ where $\Delta$ does not vanish is a nonempty Zariski open subset, and therefore dense in $M_n(\F)^g$. 
In the case $\F\in\{\RR,\CC\}$, 
$S$ is also dense  in the Euclidean topology.
By the definition of $S$ it follows that every $S$-generic vector subspace of $M_n(\F)$ is of length $O(\log n)$.
\end{proof}

\begin{proof}[Proof of Corollary \ref{Ros}]
Choose $A_1,\dots,A_{\ol g}$ such that $(A_1,\dots,A_{\ol g})\in S\subseteq M_n(\F)^{\ol g}$,
  where $\ol g$ is the smallest among $m\in \N$ satisfying 
$m^{d}\geq n$, 
and $S$ is the set described in the proof of Corollary \ref{generic}. 
For the remaining $A_{\ol g+1},\dots,A_g$ arbitrary $n\times n$ matrices will do. 
\end{proof}

\begin{remark}
(a) It is not difficult to see that one can take 
the matrices in Corollary \ref{Ros} to be
symmetric.
One only needs to note that the proof of Lemma \ref{graf-g} also works for the undirected version of the graphs $G_d^g$ and then use these in the proof of Theorem \ref{thm:words}.

(b)
The proof of Proposition \ref{prop} also leads to an explicit construction of 
$n\times n$ matrices such that words of degree $2d=2\lceil \log_g n\rceil$ in these matrices span $M_n(\F)$. 
We give  an example in characteristic $0$. 
Keep the notation from the proof of Corollary \ref{Ros}.
Let $M=n!(n^{2d-1})^n$. 
We set all variables that do not appear in $m_n$ to zero, and denote by $\C'$ the polynomial algebra in the remaining variables. 
Let us order the variables as follows:
$x^{(k)}_{i,i+(k-1)\bar g^{s-1}}<x^{(\ell)}_{j,j+(k-1)\bar g^{t-1}}$, 
(resp. $x^{(k)}_{i,i+(k-1)\bar g^{s-1}}<x^{(\ell)}_{j+(k-1)\bar g^{t-1},j})$, if 
$(s,k,i)<(t,\ell,j)$ (resp. $(s,k,i)\leq (t,\ell,j)$) in the lexicographic order, 
and take the corresponding lexicographic ordering on $\C'$.
Let 
$$c_1=3, \quad c_s=2\bar g^{s-1}(\bar g-1)+\bar g^{s-2}\; \;(s>1), \quad c=\sum_{s=1}^d c_s.$$ 
We further define for $g^{s-2}<i\leq \bar g^{s-1}$ and $1\leq j\leq \bar g^{s-1}$,
\begin{gather*}
f_{1,s,i^+}=f_{1,s,i^-}=\sum_{t=0}^{s-1} c_t+j \;\;(\bar g^{s-2}<j\leq \bar g^{s-1}),\\
 f_{k,s,j^+}=\sum_{t=0}^{s-1}c_t+\bar g^{s-2}+2\bar g^{s-1}(k-2)+j, \quad f_{k,s,j^-}=\sum_{t=0}^{s-1}c_t+\bar g^{s-2}+2\bar g^{s-1}(k-2)+\bar g^{s-1}+j.
 \end{gather*}
We set 
$$A^{(k)}_{i,i+(k-1)\bar g^{s-1}}=M^{2d(c-f_{k,s,i^+})},$$ 
$$A^{(k)}_{i+(k-1)\bar g^{s-1},i}=M^{2d(c-f_{k,s,i^-})}.$$ 
Since the monomial $m_n$ is the maximal monomial in $\C'$, 
the degree of monomials appearing in $\Delta(w_1(X_1,\dots,X_g),\dots,w_{n^2}(X_1,\dots,X_{\bar g}))$ is $2d$, and there appear at most 
$M$ monomials in $\Delta$ (counted with multiplicity).
It is easy to see that the constructed $A^{(k)}$, $1\leq k\leq  \bar g$, and 
arbitrary $A^{(k)}$, $\bar g<k\leq g$,  have the desired property of Corollary \ref{Ros}. 

(c)
It would be interesting to know whether {\em arbitrary} $n^2$ words in $x,y$ of fixed length $d\geq \lceil{2\log_2n}\rceil$ sweep $M_n(\F)$. 
If the answer were positive then we could deduce that a quasi-identity of $M_n(\F)$ (see \cite{BPS} for the definition) $\sum_M \lambda_M M$ with $\deg M=d$ cannot be a sum of fewer than $n^2$ monomials, and this bound is sharp. This should be seen in contrast with \cite[Exercise 7.2.3]{Row}, stating that a multilinear polynomial identity of $M_n(\F)$ cannot be a sum of fewer than $2^n$ monomials. However, the sharp bound is to the best of our knowledge not known.
\end{remark}

\end{document}